\documentclass[10pt]{article}
\usepackage{amsmath,amsthm,amsfonts,amssymb,amscd,enumerate,graphicx}

\theoremstyle{plain}
\newtheorem{SConjecture}[subsection]{Singer's Conjecture}
\newtheorem{subSCConjecture}[subsection]{Singer's Conjecture for Coxeter groups}
\newtheorem{MTheorem}[subsection]{The Main Theorem}

\newtheorem{Proposition}[subsection]{Proposition}

\newtheorem{Cor}[subsection]{Corollary}

\newtheorem{Lem}[subsection]{Lemma}

\theoremstyle{definition}

\newtheorem{Remark}[subsection]{Remark}
\newtheorem{Example}[subsection]{Example}

\newcommand{\Card}{\operatorname{Card}}
\newcommand{\wh}{\widehat}

\newcommand{\cs}{\mathcal{S}}

\newcommand{\cf}{\mathcal{F}}
\newcommand{\cH}{\mathcal{H}}

\newcommand{\Ltwo}{L^2}

\def\l{\operatorname{\ell}}
\newcommand{\ltwo}{\l^2}

\newcommand{\gS}{\Sigma}

\newcommand{\gO}{\Omega}

\newcommand{\gs}{\sigma}

\newcommand{\BS}{\mathbb{S}}

\newcommand{\BR}{\mathbb{R}}

\newcommand{\BN}{\mathbb{N}}

\numberwithin{equation}{section}
\begin{document}

\title{The $\ltwo$-homology of even Coxeter groups}

\author{Timothy A. Schroeder}

\date{July 17, 2007}
\maketitle

\begin{abstract}
Given a Coxeter system $(W,S)$, there is an associated CW-complex, denoted $\gS(W,S)$ (or simply $\gS$), on which $W$ acts properly and cocompactly.  This is the Davis complex.  $L$, the nerve of $(W,S)$, is a finite simplicial complex.  We prove that when $(W,S)$ is an \textit{even} Coxeter system and $L$ is a flag triangulation of $\BS^{3}$, then the reduced $\ltwo$-homology of $\gS$ vanishes in all but the middle dimension.  In so doing, our main effort will be examining a certain subspace of $\gS$ called the $(S,t)$-\emph{ruin}, for some $t\in S$.  To calculate the $\ltwo$-homology of this ruin, we subdivide a component of this ruin into subcomplexes we call \emph{colors} and then employ a series of  Mayer-Vietoris arguments, taking the union of these colors.  Once we have established the $\ltwo$-homology of the $(S,t)$-ruin, we will be able to calculate that of $\gS$.
\end{abstract}

\section{Introduction}\label{s:intro}
The following conjecture is attributed to Singer.

\begin{SConjecture}\label{conj:singer} If $M^{n}$ is a closed aspherical manifold, then the reduced $\ltwo$-homology of $\widetilde{M}^{n}$, $\cH_{\ast}(\widetilde{M}^n)$, vanishes for all $\ast\neq\frac{n}{2}$.
\end{SConjecture}

Singer's conjecture holds for elementary reasons in dimensions $\leq 2$.  Indeed, top-dimensional cycles on manifolds are constant on each component, so a square-summable cycle on an infinite component is constant $0$.  As a result, Conjecture \ref{conj:singer} in dimension $\leq 2$ follows from Poincar\'e duality.  In \cite{LL}, Lott and L\"uck proved that it holds for those aspherical $3$-manifolds for which Thurston's Geometrization Conjecture is true.  (Hence, by Parelman, all aspherical $3$-manifolds.)  For details on $\ltwo$-homology theory, see \cite{davismoussong}, \cite{do2} and \cite{eckmann}.  

Let $S$ be a finite set of generators.  A \emph{Coxeter matrix} on $S$ is a symmetric $S\times S$ matrix $M=(m_{st})$ with entries in $\BN\cup\{\infty\}$ such that each diagonal entry is $1$ and each off diagonal entry is $\geq 2$.  The matrix $M$ gives a presentation for an associated \emph{Coxeter} group $W$:
\begin{equation}\label{e:coxetergroup}
	W=\left\langle S\mid (st)^{m_{st}}=1, \text{ for each pair } (s,t) \text{ with } m_{st}\neq\infty\right\rangle.
\end{equation}
The pair $(W,S)$ is called a \emph{Coxeter system}.  Denote by $L$ the nerve of $(W,S)$.  In several papers (e.g., \cite{davisannals}, \cite{davisbook}, and \cite{davismoussong}), M. Davis describes a construction which associates to any Coxeter system $(W,S)$, a simplicial complex $\gS(W,S)$, or simply $\gS$ when the Coxeter system is clear, on which $W$ acts properly and cocompactly.  The two salient features of $\gS$ are that $(1)$ it is contractible and $(2)$ it admits a cellulation under which the nerve of each vertex is $L$.  It follows that if $L$ is a triangulation of $\BS^{n-1}$, $\gS$ is an $n$-manifold.  There is a special case of Singer's conjecture for such manifolds.  
 
\begin{subSCConjecture}\label{conj:singerc} Let $(W,S)$ be a Coxeter system such that its nerve, $L$, is a (weighted) triangulation of $\BS^{n-1}$.  Then 
\[\cH_{i}(\gS(W,S))=0 \text{ for all } i\neq\frac{n}{2}.\]
\end{subSCConjecture}

In \cite{do2}, Davis and Okun prove that if Conjecture \ref{conj:singerc} for \emph{right-angled} Coxeter systems is true in some odd dimension $n$, then it is also true for right-angled systems in dimension $n+1$.  (A Coxeter system is right-angled if generators either commute or have no relation.)  They also show that Thurston's Geometrization Conjecture holds for these Davis $3$-manifolds arising from right-angled Coxeter systems.  Hence, the Lott and L\"uck result implies that Conjecture \ref{conj:singerc} for right-angled Coxeter systems is true for $n=3$ and, therefore, also for $n=4$.  (Davis and Okun also show that Andreev's theorem, \cite[Theorem 2]{andreev2}, implies Conjecture \ref{conj:singerc} in dimension $3$ for right-angled systems.  In fact, using methods similar to those in \cite{do2}, one can show that Andreev's theorem implies \ref{conj:singerc} for arbitrary Coxeter systems.) 

Right-angled Coxeter systems are specific examples of \emph{even} Coxeter systems.  We say a Coxeter system is even if for any two generators $s\neq t$, $(st)^{m_{st}}=1$ implies that $m_{st}$ is even.  The purpose of this paper is to prove the following:

\begin{MTheorem}\label{t:main} Let $(W,S)$ be an even Coxeter system whose nerve $L$ is a flag triangulation of $\BS^3$.  Then $\cH_{i}(\gS(W,S))=0$ for $i\neq 2$.
\end{MTheorem}

In order to prove Theorem \ref{t:main}, we define a certain subspace $\gO$ of $\gS$, and its boundary $\partial\gO$.  We call the pair $(\gO,\partial\gO$) a \emph{ruin}.  We then subdivide $\gO$ into subspaces we call ``boundary collars,'' which are isomorphic to $B\times\left[0,1\right]$, where $B$ is a component of $\partial\gO$.  We paint these boundary collars finitely many \emph{colors}, which can be categorized as even or odd.  The painting of $\gO$ is virtually invariant under the group action on $\gO$.  Moreover, the intersection of two even colors is $2$-acyclic and the intersection of an odd color with all the evens is acyclic.  Then using Mayer-Vietoris, we are able to prove that $\cH_{\ast}(\gO,\partial\gO)=0$ for $\ast=3,4$.   

Next, we prove that for any $V\subseteq S$, and any $t\in V$, $\cH_{\ast}(\gS(W_V,V))\cong\cH_{\ast}(\gS(W_{V-t},V-t))$, where $W_V$ is the subgroup of $W$ generated by the elements of $V$.  It follows from induction and Poincar\'e duality that \ref{t:main} is true.
 
\section{Coxeter systems and the complex $\gS$}\label{s:coxeter}

\noindent
\textbf{Coxeter systems.}
 
Given a subset $U$ of $S$, define $W_{U}$ to be the subgroup of $W$ generated by the elements of $U$.  $(W_U,U)$ is a Coxeter system.  A subset $T$ of $S$ is \textit{spherical} if $W_T$ is a finite subgroup of $W$.  In this case, we will also say that the subgroup $W_{T}$ is spherical.  We say the Coxeter system $(W,S)$ is \textit{even} if for any $s,t\in S$ with $s\neq t$, $m_{st}$ is either even or infinite.

Given $w\in W$, we call an expression $w=s_{1}s_{2}\cdots s_{n}$ \emph{reduced} if there does not exist an integer $m<n$ with $w=s'_{1}s'_{2}\cdots s'_{m}$.  Define the \emph{length of $w$}, or $l(w)$, to be the integer $n$ such that $s_{1}s_{2}\cdots s_{n}$, $s_{i}\in S$, is a reduced expression for $w$.  Denote by $S(w)$ the set of elements of $S$ which comprise a reduced expression for $w$.  This set is well-defined, \cite[Proposition 4.1.1]{davisbook}.

For $T\subseteq S$ and $w\in W$, the coset $wW_{T}$ contains a unique element of minimal length.  This element is said to be $(\emptyset, T)$-reduced.  Moreover, it is shown in \cite[Ex. 3, pp. 31-32]{bourbaki}, that an element is $(\emptyset, T)$-reduced if and only if $l(wt)>l(w)$ for all $t\in T$.  Likewise, we can define the $(T,\emptyset)$-reduced elements to be those $w$ such that $l(tw)>l(w)$ for all $t\in T$.  So given $X,Y\subseteq S$, we say an element $w\in W$ is $(X,Y)$-reduced if it is both $(X,\emptyset)$-reduced and $(\emptyset,Y)$-reduced.

\textbf{Shortening elements of $W$.} We have the so-called ``Exchange'' (\textbf{E}) condition for Coxeter systems (\cite[Ch 4. Section 1, Lemma 3]{bourbaki} or \cite[Theorem 3.3.4]{davisbook}): 
\begin{itemize}
	\item\label{i:exchange} (\textbf{E}) Given a reduced expression $w=(s_1\cdots s_k)$ and an element $s\in S$, either $\ell(sw)=k+1$ or there is an index $i$ such that 
	\[sw=(s_1\cdots\wh{s_i}\cdots s_k).\]
\end{itemize}
In the case of even Coxeter systems, the parity of a given generator in the set expressions for an element of $W$ is well-defined.  (We prove this herein, Lemma \ref{l:T-hom}.)  So, in $(\textbf{E})$, $s_i=s$; i.e, if an element of $s\in S$ shortens a given element of $W$, it does so by deleting an instance of $s$ in an expression for $w$.

It is also a fact about Coxeter groups (\cite[Theorem 3.4.2]{davisbook}) that if two reduced expressions represent the same element, then one can be transformed into the other by replacing alternating subwords of the form $(sts\ldots)$ of length $m_{st}$ by the alternating word $(tst\ldots)$ of length $m_{st}$.  The proof of the first of the following two lemmas follows immediately from this.  

\begin{Lem}\label{l:onereduction} Let $t\in S$, $w\in W_{S-t}$ and $v\in W$ with $wtv$ reduced.  If there exists an $r\in S(w)-S(v)$ with $(rt)^2\neq 1$, then all $r$'s appears to the left of $t$ in any reduced expression for $wtv$.
\end{Lem}

\begin{Lem}\label{l:reduction} Let $(W,S)$ be an even Coxeter system, let $t,s\in S$ be such that $2<m_{st}<\infty$ and let $U_{st}=\{r\in S\mid m_{rt}=m_{rs}=2\}$.  Suppose that $tstw'=wtv$ (reduced) where $w'\in W$, $w\in W_{S-t}$ and $S(v)\subset U_{st}\cup\{s,t\}$.  Then $S(w)\subseteq U_{st}\cup\{s\}$. 
\end{Lem}
\begin{proof} Suppose that $w$ is a counterexample of minimum length.  $w$ cannot start with an element of $U_{st}$, since if it did, multiplication on the left by this element would produce a shorter counterexample.  Nor can $w$ begin with $s$, since by the exchange condition, multiplication on the left by $s$ would cancel an $s$ in $w'$, producing a shorter counterexample.  Therefore, $w$ must start with some $r$ which either does not commute with $t$ or does not commute with $s$.  By minimality we may also assume that every element appearing after $r$ in $w$ is from $U_{st}\cup\{s\}$.  

If $r$ does not commute with $t$, then by \ref{l:onereduction}, $r$ appears to the left of $t$ in any reduced expression for $wtv$; a contradiction to $tstw'=wtv$.  If $r$ does commute with $t$ but does not commute with $s$, then multiply both sides of $tstw'=wtv$ by $t$ leaving $stw'=w''sv'$ (reduced) where $w''$ begins with $r$, $S(v')\in U_{st}\cup\{s,t\}$ and $s\notin S(w'')$.  Then, with $t$ in \ref{l:onereduction} replaced by $s$, we have that $r$ appears to the left of $s$ in any reduced expression for $wtv$; a contradiction to $stw'=w''sv'$.  
\end{proof}   

\noindent
\textbf{The complex $\gS$.}
 
Let $(W,S)$ be an arbitrary Coxeter system.  Denote by $\cs$ the poset of spherical subsets of $S$, partially ordered by inclusion.  Given a subset $V$ of $S$, let $\cs_{<V}:=\{T\in \cs|T\subset V\}$.  Similar definitions exist for $>, \leq,\geq$.  For any $w\in W$ and $T\in \cs$, we call the coset $wW_{T}$ a \emph{spherical coset}.  The poset of all spherical cosets we will denote by $W\cs$.  

The poset $\cs_{>\emptyset}$ is an abstract simplicial complex, denote it by $L$, and call it the \emph{nerve} of $(W,S)$.  The vertex set of $L$ is $S$ and a non-empty subset of vertices $T$ spans a simplex of $L$ if and only if $T$ is spherical.  

Let $K=|\cs|$, the geometric realization of the poset $\cs$.  It is the cone on the barycentric subdivision of $L$, the cone point corresponding to the empty set, thus a finite simplicial complex.  Denote by $\gS(W,S)$, or simply $\gS$ when the system is clear, the geometric realization of the poset $W\cs$.  This is the Davis complex.  The natural action of $W$ on $W\cs$ induces a simplicial action of $W$ on $\gS$ which is proper and cocompact.  $K$ includes naturally into $\gS$ via the map induced by $T \rightarrow W_{T}$, so we view $K$ as a subcomplex of $\gS$ and note that it is a strict fundamental domain for the action of $W$ on $\gS$.  

For any element $w\in W$, write $wK$ for the $w$-translate of $K$ in $\gS$.  Let $w,w'\in W$ and consider $wK\cap w'K$.  This intersection is non-empty if and only if $V=S(w^{-1}w)$ is a spherical subset.  In fact, $wK\cap w'K$ is simplicially isomorphic to $|\cs_{\left[V,T\right]}|$, the geometric realization of $\cs_{\left[V,T\right]}:=\{V'\in\cs\mid V\subseteq V'\subseteq T\}$.

\textbf{A cubical structure on $\gS$.}  For each $w\in W$, $T\in\cs$, denote by $w\cs_{\leq T}$ the subposet $\{wW_V\mid V\subseteq T\}$ of $W\cs$.  Put $n=\Card(T)$.  $|w\cs_{\leq T}|$ has the combinatorial structure of a subdivision of an $n$-cube.  We identify the sub-simplicial complex $|w\cs_{\leq T}|$ of $\gS$ with this coarser cubical structure and call it a \emph{cube of type $T$}.  Note that the vertices of these cubes correspond to spherical subsets $V\in\cs_{\leq T}$.   (For details on this cubical structure, see \cite{moussong}.)

\textbf{A cellulation of $\gS$ by Coxeter cells.}  $\gS$ has a coarser cell structure: its cellulation by ``Coxeter cells.''  (For reference, see \cite{davisbook},\cite{do2}, and \cite{ddjo}.)  Suppose that $T\in \cs$; then by definition $W_{T}$ is finite.  Take the canonical representation of $W_{T}$ on $\BR^{\Card(T)}$ and choose a point $x$ in the interior of a fundamental chamber.  The \emph{Coxeter cell of type $T$} is defined as the convex hull $C$, in $\BR^{\Card(T)}$, of $W_{T}x$ (a generic $W_{T}$-orbit).  The vertices of $C$ are in 1-1 correspondence with the elements of $W_{T}$.  Furthermore, a subset of these vertices is the vertex set of a face of $C$ if and only if it corresponds to the set of elements in a coset of the form $wW_{V}$, where $w\in W_{T}$ and $V\subset T$.  Hence, the poset of non-empty faces of $C$ is naturally identified with the poset $W_{T}\cs_{\leq T}:= \{wW_{V}|w\in W_{T}, V\subset T\}$.  Therefore, we can identify the simplicial complex $\gS(W_{T},T)$ with the barycentric subdivision of the Coxeter cell of type $T$. 

Now, for each $T\in \cs^{(k)}$ and $w\in W$, the poset $W\cs_{\leq wW_{T}}$ is isomorphic to the poset $W_T\cs_{\leq T}$ via the map $vW_{V}\rightarrow w^{-1}vW_{V}$.  Thus, the subcomplex of $\gS(W,S)$ which is obtained from the poset $W\cs_{\leq wW_{T}}$ may be identified with the barycentric subdivision of the $k$-cell of type $T$.  In this way, we put a cell structure on $\gS$ which is coarser than the simplicial structure by identifying each simplicial subcomplex $|W\cs_{\leq wW_{T}}|$ with a cell of type $T$.

We will write $\gS_{cc}$, when necessary, to denote the Davis complex equipped with this cellulation by Coxeter cells.  Under this cellulation, the $0$-cells of $\gS_{cc}$ correspond to cosets of $W_{\emptyset}$, i.e. to elements from $W$; and $1$-cells correspond to cosets of $W_{s}$, $s\in S$.  The features of this cellulation are summarized by the following, from \cite{davisbook}.

\begin{Proposition}\label{p:coxeter} There is a natural cell structure on $\gS$ so that 
\begin{itemize}
\item its vertex set is $W$, its 1-skeleton is the Cayley graph of $(W,S)$ and its 2-skeleton is a Cayley 2-complex.
\item each cell is a Coxeter cell.
\item the link of each vertex is isomorphic to $L$ (the nerve of $(W,S)$) and so if $L$ is a triangulation of $\BS^{n-1}$, $\gS$ is a topological $n$-manifold.
\item a subset of $W$ is the vertex set of a cell if and only if it is a spherical coset and
\item the poset of cells is $W\cs$.
\end{itemize}
\end{Proposition}

\noindent
\textbf{Ruins.}

The following subspaces are defined in \cite{ddjo}.  Let $(W,S)$ be a Coxeter system.  For any $U\subseteq S$, let $\cs(U)=\{T\in\cs|T\subset U\}$ and let $\gS(U)$ be the subcomplex of $\gS_{cc}$ consisting of all cells of type $T$, with $T\in\cs(U)$. 

Given $T\in \cs(U)$, define three subcomplexes of $\gS(U)$:
\begin{align}
\gO (U,T):\quad & \text{the union of closed cells of type $T'$, with $T'\in 
\cs(U)_{\geq T}$,}\notag\\
\wh{\gO}(U,T):\quad & \text{the union of closed cells of type $T''$}, 
T''\in \cs(U), T''\notin \cs (U)_{\geq T},\notag\\
\partial\gO (U,T):\quad & \text{the cells of $\gO (U,T)$ of type $T''$,
with } T''\notin \cs (U)_{\geq T}.\notag
\end{align}
The pair $(\gO (U,T),\partial\gO (U,T))$ is called the  \emph{$(U,T)$-ruin}.
For $T=\emptyset$, we have $\gO (U,\emptyset)=\gS (U)$ and
$\partial\gO (U,\emptyset)=\emptyset$.

\textbf{The subspace $\gO$.}  Let $t\in S$.  We call the $(S,t)$-ruin a \emph{one-letter ruin}.  Put $U:=\{s\in S\mid m_{st}<\infty\}$.  The path components of $\gO(S,t)$ are indexed by the cosets $W/W_{U}$.  Denote by $\gO$ the path-component of $\gO(S,t)$ with vertex set corresponding $W_{U}$.  The action of $W_{U}$ on $\gS$ restricts to an action on $\gO$.  Let $\partial\gO:=\gO\cap \partial\gO(S,t)$ and put $K(U):=K\cap\gO$.  Note that the $W_{U}$-translates of $K(U)$ cover $\gO$, i.e. $\gO=\bigcup_{w\in W_{U}}wK(U)$.

If we restrict our attention to cubes of type $T$, where $T\subseteq T'$ for some $T'\in\cs_{\geq t}$, $\gO$ is a cubical complex and $\partial\gO$ is a subcomplex.  Moreover, if $B$ is a component of $\partial\gO$, the space $D:=B\times\left[0,1\right]$ is isomorphic to the union of the $w$-translates of $K(U)$ where $w$ is a vertex of $B$.  We call such subspaces \emph{boundary collars}.  It is clear that the collection of boundary collars covers $\gO$.  We denote by $\partial_{in}(D)$ the ``1-end'' of this product and note that it is comprised of $0$-simplices corresponding to elements of $\cs_{\geq t}$.  The boundary collars intersect along these ``inner'' boundaries.  

\section{The $\ell^2$-homology of $\gO(S,t)$}

Here and for the remainder of the paper, we require that $(W,S)$ be an even Coxeter system whose nerve $L$ is a flag triangulation of $\BS^3$.  Fix $t\in S$ and let $U$, $\gO$ and $\partial\gO$ be defined as in \ref{s:coxeter}.  

Any $s\in U$ has the property that $m_{st}<\infty$.  Let $S':=\{s\in U\mid m_{st}>2\}$, and assume that $S'$ is not empty.  The group $W_{U}$ has the following properties.

\begin{Lem}\label{l:srel} Suppose that $L$ is flag.  Then for $s,s'\in S'$, either $s=s'$, or $m_{ss'}=\infty$.
\end{Lem}
\begin{proof} Suppose that $s\neq s'$ and that $m_{ss'}<\infty$.  Then $\{s,s'\}\in \cs$, and since $s,s'$ are both in $U$, the vertices corresponding to $s$, $s'$ and $t$ are pairwise connected in $L$. $L$ is a flag complex, so this implies that $\{s,s',t\} \in \cs$.  But
\[ \frac{1}{m_{ss'}}+\frac{1}{m_{st}}+\frac{1}{m_{ts'}}\leq \frac{1}{m_{ss'}}+\frac{1}{4}+\frac{1}{4}\leq 1.\]
This contradicts $\{s,s',t\}$ being a spherical subset.  So we must have that $m_{ss'}=\infty$.
\end{proof}

\begin{Cor}\label{c:commute} Let $s\in S'$ and let $T\in \cs_{\geq\{s,t\}}$.  Then $m_{ut}=m_{us}=2$ for $u\in T-\{s,t\}$.  In other words, the generators from $T-\{s,t\}$ commute with both $s$ and $t$.
\end{Cor}

Let $L_{st}$ denote the link in $L$ of the edge corresponding to the vertices $s$ and $t$.  The above Corollary states that the generators corresponding to the vertex set of $L_{st}$ commute with both $s$ and $t$.  Denote this set of generators by $U_{st}$.  

Of particular interest to us will be elements of $W_U$ with a reduced expression of the form $tst\cdots st$ for some $s\in S'$.  Since $W$ is even, this expression is unique, and we have the following Lemma.  

\begin{Lem}\label{l:XY-reduced} Let $s\in S'$ and let $u\in W_{\{s,t\}}$ be such that $u=tst\cdots st$, is a reduced expression beginning and ending with $t$.  Then $u$ is $(U-t,U-t)$-reduced.
\end{Lem}

\begin{Lem}\label{l:T-hom}  Let $V,T\subset S$ and consider the function $g_{VT}:W_{V}\rightarrow W_{T}$ induced by the following rule: $g_{VT}(s)=s$ if $s\in V\cap T$ and $g_{VT}(s)=e$ (the identity element of $W$) for $s\in V-T$.  $g_{VT}$ is a homomorphism.
\end{Lem}
\begin{proof}  We show that $g_{VT}$ respects the relations in $W_{V}$.  Let $s,u\in V$ be such that $(su)^m=1$.  Then 
\begin{equation*}
	g_{VT}((su)^m)=
	\begin{cases}
		(su)^m & \text{ if } s\in T, u\in T\\
		s^m & \text{ if } s\in T, u\notin T\\
		u^m & \text{ if } u\in T, s\notin T\\
		e & \text{ if } s\notin T, u\notin T.
	\end{cases}
\end{equation*}
In all cases, since $(W_V,V)$ is even, $g_{T}((su)^m)=e$.  
\end{proof}

Then with $T\in \cs_{\geq t} $ and $U$ as above, we define an action of $W_U$ on the set of cosets $W_T/W_{T-t}$: For $w\in W_{U}$ and $v\in W_{T}$, define 
\begin{equation}\label{e:action}
	w\cdot vW_{T-t}=g_{UT}(w)vW_{T-t}.
\end{equation}

\noindent
\textbf{Coloring boundary collars.}

Set
\begin{equation*}
	A=\prod_{T\in\cs_{\geq t}}W_{T}/W_{T-t}.
\end{equation*}
We call $A$ the set of colors, note that it is a finite set.  The action defined in (\ref{e:action}) extends to a diagonal $W_{U}$-action on $A$.  So for $w\in W_{U}$ and $a\in A$, write $w\cdot a$ to denote $w$ acting on $a$.  Let $\bar{e}$ be the element of $A$ defined by taking the trivial coset $W_{T-t}$ for each $T\in \cs_{\geq t}$.  Vertices of $\gO$ correspond to group elements of $W_{U}$, so we paint the vertices of $\gO$ by defining a map $c:W_{U}\rightarrow A$ with the rule $c(w):=w\cdot\bar{e}$.

\begin{Remark}\label{r:trivial} If an element $w\in W_{U}$ contains no $t$'s in any of its reduced expressions, then $w$ acts trivially on the element $\bar{e}$, i.e. $w\cdot\bar{e}=\bar{e}$.
\end{Remark}  

We will paint the space $wK(U)$ with $c(w)$.  In this way, all of $\gO$ is colored by some element of $A$.  For vertices $w$ and $w'$ of the same component $B$ of $\partial\gO$, $h=w^{-1}w'\in W_{U-t}$.  So $c(w')=c(wh)=wh\cdot\bar{e}=w\cdot\bar{e}=c(w)$, and therefore all of $D=B\times\left[0,1\right]$ is painted with $c(w)$.  Note that each component of $\partial\gO$ is monochromatic while $\partial_{in}(D)$ is not.

\begin{Lem}\label{l:samecolordisjoint} Let $D=B\times\left[0,1\right]$ and $D'=B'\times\left[0,1\right]$ be boundary collars where $B$ and $B'$ are different components of $\partial\gO$. Suppose that the vertices of $B$ and $B'$ have the same color.  Then $D\cap D'=\emptyset$.
\end{Lem}
\begin{proof} Suppose, by way of contradiction, that $D\cap D'\neq\emptyset$, i.e. there exist vertices $w\in B$, $w'\in B'$ such that $c(w)=c(w')$ and $wK(U)\cap w'K(U)\neq\emptyset$.  Let $V=S(w^{-1}w')$ and $v=w^{-1}w'$.  Then $c(w)=c(w')\Rightarrow w\cdot\bar{e}=wv\cdot\bar{e}\Rightarrow \bar{e}=v\cdot\bar{e}$. Thus, for any $T\in \cs_{\geq t}$, we have that 
\begin{equation}\label{e:h}
	v\cdot W_{T-t}=W_{T-t}.
\end{equation}
$V\cup t$ is spherical, and since $v\in W_{V}$, the action of $v$ on $W_{V\cup t}/W_{V-t}$ defined in (\ref{e:action}) is left multiplication by $v$.  So by (\ref{e:h}), we have that $v\in W_{V-t}$.  But this contradicts $w$ and $w'$ coming from different components of $\partial\gO$.
\end{proof}

Then for $c\in A$, define the \emph{$c$-collars}, $F_{c}$, to be the disjoint union of the boundary collars $D=B\times\left[0,1\right]$ where each component $B$ of $\partial\gO$ has the color $c$.  We refer to these collections as \emph{colors}.  The collection of colors is a finite cover of $\gO$.  

\noindent
\textbf{Even and odd colors.}  

Let $T=\{t\}$ and consider the homomorphism $g_{UT}:W_{U}\rightarrow W_{t}$ defined in (\ref{l:T-hom}).  Under $g_{UT}$, an element $w\in W_{U}$ is sent to the identity in $W_{t}$ if $w$ has an even number of $t$'s present in some  factorization (and therefore, all factorizations) as a product of generators from $U$ and an element $w\in W_U$ is sent to $t\in W_{t}$ if $w$ has an odd number of $t$'s present in some factorization.  Thus, we call a vertex $w$ \emph{even} if $g_{UT}(w)=e$; \emph{odd} if $g_{UT}(w)=t$.  If two vertices $w$ and $w'$ are such that $c(w)=c(w')$, then clearly $g_{UT}(w)=g_{UT}(w')$.  So we may also classify the colors (both the elements of $A$ and the collections of boundary collars), as even or odd.  We will suppress the subscript $c$ and say a color $F$ is even or odd.  

Of fundamental importance will be how these colors intersect.  By Remark \ref{r:trivial}, we know that in order for the vertices of a Coxeter cell to support two different colors, this cell must be of type $T\in\cs_{\geq t}$.  But, for a cell to support two different \emph{even} vertices, $v$ and $v'$, this cell must be of type $T\in\cs_{\geq\{s,t\}}$ for exactly one $s\in S'$ (uniqueness is given by Corollary \ref{c:commute}).  Moreover, $w=v^{-1}v'$ has the properties that $\{s,t\}\subseteq S(w)$ and that it contains at least two, and an even number of $t$'s in any factorization as a product of generators.  We call such $w$ \emph{$t$-even}.  

\begin{Example}\label{ex:dim2} The following example is representative of our situation.  Suppose $L=\BS^1$, and $U=\{t,r,s\mid (rt)^2=1, (st)^4=1\}$.  $\gO$ is represented in Figure \ref{fig:ruincolors9}.  The black dots represent the vertices of the Coxeter cellulation, with the vertices $e$ and $tst$ labeled.  The even colors are shaded.  Even boundary collars intersect in a $0$-simplex corresponding to the spherical subset $\{s,t\}$.  The intersection of one odd color and all evens is the inner boundary of the odd color.  
\end{Example} 
\begin{figure}[htbp]
	\centering
		\includegraphics[width=8cm]{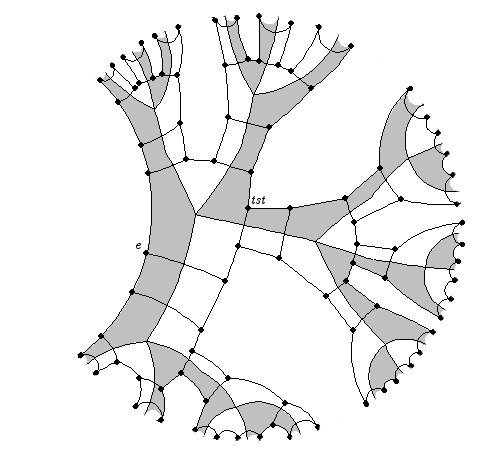}
	\caption{Even and Odd Colors of $\gO$}
	\label{fig:ruincolors9}
\end{figure}

\textbf{The intersection of even colors.}  Let $D_0$ denote the boundary collar containing the vertex $e$.  Fix $s\in S'$ and let $D_2$ denote the boundary collar containing the vertex $u$, where $u\in W_{\{s,t\}}$ is $t$-even and has a reduced expression ending in $t$.  We study $D_0\cap D_2$.  

\begin{Lem}\label{l:W'-orbit} Let $W':=W_{U_{st}}$ and let $K'=K(U)\cap uK(U)$.  Denote by $W'K'$ the orbit of $K'$ under $W'$.  $D_{0}\cap D_{2} = W'K'$.
\end{Lem}
\begin{proof} For any $w\in W'$, the vertex $w$ is in the same component of $\partial\gO$ as $e$, and therefore $wK(U)\subset D_0$.  $wu=uw$, so $wu$ is in the same component of $\partial\gO$ as $u$ and $wuK(U)\subset D_{2}$.  Thus $wK'=wK(U)\cap wuK(U)\subset D_{0}\cap D_{2}$.

Now let $\gs$ be a $0$-simplex in $D_{0}\cap D_{2}$.  Then there exist $w,w'\in W_{U-t}$ such that $\gs\in wK(U)\cap uw'K(U)$, i.e. $\gs$ is simultaneously the $w$- and $uw'$-translate of a $0$-simplex $\gs'$ in $K(U)$.  Let $V$ be the spherical subset to which $\gs'$ corresponds and let $v\in W_V$ be such that $uw'=wv$.  $c(e)=c(w)$ and $c(u)=c(uw')$, so $w$ and $uw'$ are differently colored even vertices of a cell of type $V$.  Thus $\{s',t\}\subseteq S(v)\subseteq V$ for exactly one $s'\in S'$ and $v$ is $t$-even.  

\textbf{Claim 1}: $s'=s$.\\  
\textbf{Pf of Claim 1}: Since $w'\in W_{U-t}$, $c(u)=c(uw')=c(wv)$, i.e. $u$ and $wv$ act the same on every coordinate of $\bar{e}$.  Consider the $\{s,t\}$-coordinate.  $u\in W_{\{s,t\}}$ is $t$-even, so $u\cdot W_{s}=uW_{s}$ and $uW_{s}\neq W_{s}$.  But if $s\notin S(v)$, then $v$ being $t$-even and $w\in W_{U-t}$ imply that $wv\cdot W_{s}=W_{s}$; which contradicts $u$ and $wv$ having the same color.  So \textbf{Claim 1} is true, and as a result $V\in\cs_{\geq\{s,t\}}$ and $\gs'\in K'$.  It remains to show that $\gs$ is in the $W'$-orbit of $K'$.  

\textbf{Claim 2}: $S(w)\subseteq (U_{st}\cup \{s\})$.\\
\textbf{Pf of Claim 2}: Take a reduced expression for $u$ which ends in $t$.  If this expression begins with $s$, multiply $u$ on the left by $s$, so that we have $suw'=swv$.  The only change this can effect on $S(w)$ is by either adding or subtracting an $s$, which is inconsequential to our claim.  So, we may assume that $u$ has a reduced expression of the form $tst\cdots st$ as described in Lemma \ref{l:XY-reduced}.  Hence, $u$ is $(U-t,U-t)$-reduced and $uw'$ has a reduced expression beginning with the subword $tst$.  $v$ is $t$-even, so $wv$ has a reduced expression of the form $wtv'$ where $w\in W_{U-t}$ and $S(v')\subset U_{st}\cup\{s,t\}$.  \textbf{Claim 2} then follows from Lemma \ref{l:reduction}.  

We now finish the proof of \ref{l:W'-orbit}.  If $s\notin S(w)$, then $w\in W'$ and we are done since $\gs$ is the $w$-translate of $\gs'$.  If $s\in S(w)$, then $w$ may be written as $qs$, with $q\in W'$ and since $s\in V$, $qsW_{V}=qW_{V}$.  So $\gs$ is also the $q$-translate of $\gs'$.
\end{proof}

\begin{Proposition}\label{p:W'-Daviscpx} $(D_{0}\cap D_{2})\cong\gS(W',U_{st})$, an infinite connected $2$-manifold.
\end{Proposition}
\begin{proof} Since $S(u)=\{s,t\}$, $K'$ is the geometric realization of the poset 
$\cs_{\geq\{s,t\}}=\{V\in \cs|\{s,t\}\subseteq V\}.$  By Lemma \ref{l:W'-orbit}, $(D_{0}\cap D_{2})\cong |W'\cs_{\geq\{s,t\}}|$, and by Corollary \ref{c:commute}, $\cs_{\geq\{s,t\}}$ is isomorphic to $\cs(U_{st})$ via the map $T\rightarrow T-\{s,t\}$.  So $(D_{0}\cap D_{2})\cong |W'\cs(U_{st})|=\gS(W',U_{st})$. 

Simplices in $L_{st}$ correspond to spherical subsets $T\in\cs$ such that neither $s$ nor $t$ is contained in $T$ but $T\cup\{s,t\}\in \cs$.  So by Corollary \ref{c:commute}, the vertex set of a simplex of $L_{st}$ corresponds to a spherical subset of $\cs(U_{st})$.  Conversely, given a spherical subset $T\in\cs(U_{st})$, $W_{T\cup\{s,t\}}=W_{T}\times W_{\{s,t\}}$, which is finite.  So $T$ corresponds to a simplex of $L_{st}$.  Thus, $L_{st}$ is the nerve of the system $(W',U_{st})$.  Since $L$ triangulates $\BS^3$, $L_{st}$ triangulates $\BS^1$.  The result follows from Proposition \ref{p:coxeter}.
\end{proof} 

\begin{Cor}\label{c:H22evens} Let $F\neq F'$ be even colors.  Then $\cH_{2}(F\cap F')=0$.
\end{Cor}
\begin{proof} Suppose that $F\neq F'$ are both even colors such that $F\cap F'\neq\emptyset$.  Then there exist even vertices $v$ and $v'$ with $vK(U)\cap v'K(U)\neq\emptyset$.  Let $w=v^{-1}v'$ and put $T=S(v^{-1}v')$.  $T$ is a spherical subset, and $v$ and $v'$ are both vertices of a cell of type $T$.  So we have exactly one $s\in S'$ with $\{s,t\}\subset T$.  Factor $w$ as $w=xq$ where $x\in W_{\{s,t\}}$ is $t$-even and $q\in W_{T-\{s,t\}}$.  Now, $x$ may not have a reduced expression ending in $t$.  If it does not, then $xs$ does and it is in the same boundary collar as $x$ and $w$.  So let 
\[
u=
\begin{cases}
	x & \text{ if $x$ has a reduced expression ending in $t$},\\
	xs & \text{ otherwise}.
\end{cases}
\]
Then $vK(U)\cap v'(U)\subseteq vK(U)\cap vuK(U)$.  Act on the left by $v^{-1}$ and we are in the situation studied in \ref{l:W'-orbit} and \ref{p:W'-Daviscpx}.  So $F\cap F'$ is the disjoint union of infinite connected $2$-manifolds.  As a result, any $2$-cycle must be constant $0$. 
\end{proof}

\begin{Remark}\label{r:oneevencolor} If $S'=\{s\in S| 2<m_{st}<\infty\}=\emptyset$, then $W_U=W_{U-t}\times W_t$ and there is one even color and one odd color.  
\end{Remark}

\textbf{Multiple even colors.}  Suppose that $D_{1}, D_{2},\ldots,D_{n}, D_{e}$ are even boundary collars.  Then 
\begin{equation*}
D_{e}\cap (\bigcup^{n}_{j=1} D_{j})= (D_{e}\cap D_{1})\cup\cdots\cup (D_{e}\cap D_{n}),
\end{equation*}
and suppose that for some $1\leq i< k\leq n$ we have that $(D_{e}\cap D_{i})\cup (D_{e}\cap D_{k})$ is not disjoint.  Let $\gs$ be a $0$-simplex contained in $D_{e}\cap D_{i}\cap D_{k}$ corresponding to a coset of the form $vW_{T}$.  Then there exists $w,w'\in W_{T}$ such that $v\in D_{e}$, $vw\in D_{i}$, $vw'\in D_{k}$ and $\gs\in vK(U)\cap vwK(U)\cap vw'K(U)$.  These three vertices are differently colored even vertices of a cell of type $T$, so $\{s,t\}\subseteq T$ for exactly one $s\in S'$ and both $w$ and $w'$ are $t$-even.  Then, as in the proof of \ref{c:H22evens}, it follows that $D_{e}\cap D_{i}=D_{e}\cap D_{k}\cong |W'\cs_{\geq\{s,t\}}|$.  So Corollary \ref{c:H22evens} generalizes to the following: 

\begin{Cor}\label{c:H2evens} Let $F_{1}$, $F_{2},\ldots, F_{n}, F_{e}$ be even colors.  Then 
\begin{equation*}
 \cH_{2}(F_{e}\cap (\bigcup^{n}_{j=1} F_{j}))=0.
\end{equation*}
\end{Cor}

\begin{Lem}\label{l:oneodd} Let $\cf_{E}$ denote the union of all even colors and let $F_o$ be an odd color.  Define   
\[\partial_{in}(F_{c}):=\coprod_{D\subset F_{c}}\partial_{in}(D).
\]
$F_{o}\cap \cf_{E}=\partial_{in}(F_{o})$. 
\end{Lem}
\begin{proof} Since $F_{o}$ is a disjoint union of boundary collars, it suffices to show that $D\cap \cf_{E}=\partial_{in}(D)$ for some boundary collar $D\subset F_{o}$.

($\supseteq$): Let $\gs$ be a $0$-simplex in $\partial_{in}(D)$.  Then $\gs$ corresponds to a coset of the form $wW_{V}$ where $V\in \cs_{\geq t}$ and $w\in W_{U}$ is an odd vertex of $D$.  Consider the even vertex $wt$.  Then since $t\in V$, $wW_{V}=wtW_{V}$, and $\gs\in wtK(U)\subset \cf_{E}$.

($\subseteq$): Now suppose that $\gs$ is a $0$-simplex contained in $D\cap \cf_{E}$.  Then there exists a spherical subset $V$ and cosets $wW_V=w'W_V$ where $w$ is odd and $w'$ is even.  Let $v=w^{-1}w'$.  Since $w$ is odd and $w'$ is even, $v$ must contain an odd number of $t$'s in any of its reduced expressions.  Therefore $t\in V$ and $\gs\in\partial_{in}(D)$.
\end{proof}

As before, let $\cf_{E}$ denote the union of all even colors, and now let $\cf_{O}$ denote the union of a sub-collection of the odd colors.  Let $\cf_{E'}=\cf_{E}\cup \cf_{O}$ and let $F_{o}$ be an odd color not in $\cf_{O}$.  Then by \ref{l:oneodd}, 
\begin{equation*}
	F_{o}\cap \cf_{E'}= (F_{o}\cap \cf_{E})\bigcup (F_{o}\cap \cf_{O})=\partial_{in}(F_{o})\bigcup (F_{o}\cap \cf_{O}).
\end{equation*}
Any $0$-simplex in $F_o$ which is also in a different color must be of the form $wW_V$, where $w$ is a vertex of $F_o$ and $V\in\cs_{\geq T}$.  Therefore $(F_{o}\cap \cf_{O})\subset \partial_{in}(F_{o})$ and $F_{o}\cap \cf_{E'}=\partial_{in}(F_{o})$.  

It is clear from the product structure on boundary collars that $\partial_{in}(F_o)\cong F_o\cap\partial\gO$, the latter a disjoint collection of components of $\partial\gO$.  Since $L$ is flag, we have a 1-1 correspondence between cells of any component of $\partial\gO$ and cells of $\gS(W_{U-t},U-t)_{cc}$.  Denote by $L_t$ the link in $L$ of the vertex corresponding to $t$, it is a triangulation of $\BS^2$ and it is isomorphic to the nerve of $(W_{U-t},U-t)$.  Then since Conjecture \ref{conj:singerc} is true in dimension $3$,
\begin{equation}\label{e:multipleodds}
	\cH_{i}(F_{o}\cap \cf_{E'})=0,
\end{equation}
for all $i$.

\begin{Proposition}\label{p:H3,4(S)=0} Let $(W,S)$ be an even Coxeter system whose nerve, $L$ is a flag triangulation of $\BS^3$.  Let $t\in S$.  Then $\cH_{\ast}(\gO(S,t),\partial\gO(S,t))=0$ for $\ast=3,4$.
\end{Proposition}
\begin{proof} We first show that $\cH_{4}(\gO,\partial\gO)=0$.  Consider the long exact sequence of the pair $(\gO,\partial\gO)$:
\[\rightarrow\cH_4(\gO)\rightarrow\cH_4(\gO,\partial\gO)\rightarrow\cH_3(\partial\gO)\rightarrow
\]
$\gO$ is a $4$-dimensional manifold with boundary, so $\cH_4(\gO)=0$ and $\cH_3(\partial\gO)=0$.  So by exactness, $\cH_4(\gO,\partial\gO)=0$.

Let $\cf_{E'}$ denote the union of a collection of even colors or the union of all evens and a collection of odd colors.  Let $F$ be a color not contained in $\cf_{E'}$ (if $\cf_{E'}$ is not all the even colors, require that $F$ be an even color).  Let $\partial_{E'}=\cf_{E'}\cap\partial\gO$ and let $\partial_F=F\cap\partial\gO$.  Note that $\partial_{E'}\cap\partial_{F}=\emptyset$ and consider the relative Mayer-Vietoris sequence of the pair $(\cf_{E'}\cup F, \partial_{E'}\cup \partial_{F})$:
\[\ldots\rightarrow\cH_{3}(\cf_{E'},\partial_{E'})\oplus\cH_{3}(F,\partial_F)\rightarrow\cH_{3}(\cf_{E'}\cup F,\partial_{E'}\cup\partial_{F})\rightarrow\cH_{2}(\cf_{E'}\cap F)\rightarrow\ldots\]
Assume that $\cH_{3}(\cf_{E'},\partial_{E'})=0$.  Each color retracts onto its boundary, so $\cH_{3}(F,\partial_{F})=0$.  If $F$ is even, then the last term vanishes by \ref{c:H2evens}, if $F$ is odd, then the last term vanishes by (\ref{e:multipleodds}).  In either case, exactness implies that $\cH_{3}(\cf_{E'}\cup F,\partial_{E'}\cup\partial_F)=0$.  It follows from induction that $\cH_{3}(\gO,\partial\gO)=0$.
\end{proof}

\section{The $\ell^2$-homology of $\gS$}

\begin{Lem}\label{l:H4-2letter} Let $V\subseteq S$ and let $T\subseteq V$ be a spherical subset with $\Card(T)=2$.  Then $\cH_{4}(\gO(V,T),\partial\gO(V,T))=0$. 
\end{Lem}
\begin{proof} If $\cs(V)^{(4)}_{>T}=\emptyset$, then $\gO(V,T)$ does not contain 4-dimensional cells, and we are done.  So assume that $\cs(V)^{(4)}_{>T}\neq\emptyset$.  The codimension 1 faces of $4$-cells of $\gO(V,T)$ are either faces of one other $4$-cell in $\gO(V,T)$ ($\gS$ is a $4$-manifold), or they are free faces, i.e they are not faces of any other $4$-cell in $\gO(V,T)$.

Suppose that cells of type $T'\in\cs(V)^{(4)}_{>T}$ have a co-dimension one face of type $F$ which is a face of another $4$-cell in $\gO(V,T)$ of type $T''$.  Then any relative $4$-cycle must be constant on adjacent cells of type $T'$ and $T''$, where $T'=R\cup \{r\}$, and $T''=R\cup \{s\}$, for some $R\in\cs(V)^{(3)}$ and $r,s\in V$.  Since $L$ is flag and $3$-dimensional, $m_{rs}=\infty$.  So in this case, there is a sequence of adjacent $4$-cells with vertex sets $W_{T'},W_{T''},sW_{T'},srW_{T''},srsW_{T'},srsrW_{T''},\ldots$.  Hence, this constant must be $0$.  

Now suppose that for a given $4$-cell of $\gO(V,T)$, every co-dimension one face is free.  This cell has faces not contained in $\partial\gO(V,T)$, so relative $4$-cycles cannot be supported on this cell.  
\end{proof}

Let $V\subseteq S$, be arbitrary; $T\subseteq V$ spherical, $\gO:=\gO(V,T)$, $\partial\gO:=\partial\gO(V,T)$.  Recall that $\gS(V)$ is the subcomplex of $\gS_{cc}$ consisting of cells of type $T'$, with $T'\subseteq V$.  We have excision isomorphisms from \cite{ddjo}:
\begin{equation}\label{e:excision1}
	C_{\ast}(\gO(V,T),\partial\gO)\cong C_{\ast}(\gS(V),\wh{\gO}(V,T)),
\end{equation}
and for any $s\in T$ and $T':=T-s$, 
\begin{equation}\label{e:excision2}
	C_{\ast}(\gS (V-s),\wh{\gO}(V-s,T'))\cong
	C_{\ast}(\wh{\gO}(V,T),\wh{\gO}(V,T')).
\end{equation}
Set $\wh{\gO}:=\wh{\gO}(V,T)$, and $\wh{\gO}':=\wh{\gO}(V,T')$.  Consider the long, weakly exact sequence of the triple $(\gS(V),\wh{\gO},\wh{\gO}')$:
\[
	\ldots\to \cH_{\ast}(\wh{\gO},\wh{\gO}')\to
	\cH_{\ast}(\gS(V),\wh{\gO}')\to
  \cH_{\ast}(\gS(V),\wh{\gO})\to\ldots  
\]
By (\ref{e:excision1}) and (\ref{e:excision2}), the left hand term excises to the homology of the $(V-s,T')$-ruin, the right hand term to that of the $(V,T)$-ruin and the middle term to that of the $(V,T')$-ruin; leaving the sequence:
\begin{equation}\label{e:ruinsequence}
	\ldots\to\cH_{\ast}(\gO(V-s,T'),\partial)\to
	\cH_{\ast}(\gO(V,T'),\partial)\to
	\cH_{\ast}(\gO(V,T),\partial)\to\ldots
\end{equation}

\begin{Proposition}\label{p:one-letter} Let $(W,S)$ be an even Coxeter system, whose nerve $L$ is a flag triangulation of $\BS^{3}$.  Let $V\subseteq S$ and $t\in V$.  Then 
\begin{equation}\label{e:one-letter}
	\cH_{\ast}(\gO(V,t),\partial\gO(V,t))=0,
\end{equation}
for $\ast=3,4$.
\end{Proposition}
\begin{proof}  It is clear that $\cH_{\ast}(\gO(V,t))=0$ for $\ast=3,4$ whenever $\Card(V)\leq 2$, so we may assume that $\Card(V)>2$.  We show (\ref{e:one-letter}) by induction on $\Card(S-V)$, Proposition \ref{p:H3,4(S)=0} giving us the base case.  Let $V=V'\cup s$ and $t\in V'$.  Assume (\ref{e:one-letter}) holds for $V$.  If $m_{st}=\infty$ then $\gO(V',t)=\gO(V,t)$ and we are done.  Otherwise, consider the sequence in (\ref{e:ruinsequence}), taking $T=\{s,t\}$, $T'=\{t\}$:
\[
\begin{array}{cccccccc}
0 & \rightarrow & \cH_{4}(\gO(V',t),\partial) & \rightarrow &  \cH_{4}(\gO(V,t),\partial) & \rightarrow &  \cH_{4}(\gO(V,\{s,t\}),\partial) &\to \\
& \rightarrow & \cH_{3}(\gO(V',t),\partial) & \rightarrow & \cH_{3}(\gO(V,t),\partial) & \rightarrow & \ldots 
\end{array}
\]
$\cH_{\ast}(\gO(V,t),\partial)=0$ for $\ast=3,4$ by assumption and $\cH_{4}(\gO(V,\{s,t\}),\partial)=0$ by \ref{l:H4-2letter}.  So by exactness,  $\cH_{4}(\gO(V',t),\partial)=0$.
\end{proof}

\begin{MTheorem}\label{t:Singer} Let $(W,S)$ be an even Coxeter system whose nerve $L$ is a flag triangulation of $\BS^{3}$ and let $\gS=\gS(W,S)$.  Then 
\[
\cH_{\ast}(\gS)=0 \text{ for } \ast\neq 2.
\]
\end{MTheorem}
\begin{proof} Let $V\subseteq S$ and $t\in V$.  Consider the following form of (\ref{e:ruinsequence}), where $T=\{t\}$:
\[
\begin{array}{ccccccccc}
	0 & \to & \cH_{4}(\gS(V-t)) & \to & \cH_{4}(\gS(V)) & \to & \cH_{4}(\gO(V,t),\partial) & \to \\
 & \to & \cH_{3}(\gS(V-t)) & \to & \cH_{3}(\gS(V)) & \to & 	\cH_{3}(\gO(V,t),\partial) & \to & \ldots 
\end{array}
\]
By Proposition \ref{p:one-letter}, $\cH_{\ast}(\gO(V,t),\partial)=0$ for $\ast=3,4$.  So by exactness,
\[
\cH_{\ast}(\gS(V-t))\cong\cH_{\ast}(\gS(V)),
\]
for $\ast=3,4$.  It follows that $\cH_{\ast}(\gS)\cong\cH_{\ast}(\gS(\emptyset))=0$ for $\ast=3,4$ and hence, by Poincar\'e duality, $\cH_{\ast}(\gS)=0$ for $\ast\neq 2$.
\end{proof}


\newpage

\begin{thebibliography}{88}

\bibitem{andreev2}
E. Andreev, \emph{On convex polyhedra of finite volume in
Loba\u{c}evski\u{i} space},  Math. USSR SB. \textbf{12} (1970), 255--259.

\bibitem{bourbaki}
N. Bourbaki, \emph{Lie Groups and Lie Algebras}, Chapters 4-6.  Springer, 2002.

\bibitem{davisannals} M. Davis, \emph{Groups generated by reflections and aspherical manifolds not covered by Euclidean space}, Ann. of Math.  117 (1983) 293-294.

\bibitem{davisbook}
M. Davis, \emph{The Geometry and Topology of Coxeter Groups}, Princeton University Press, Princeton, (2007).

\bibitem{ddjo}
M. Davis, J. Dymara, T. Januszkiewicz, and B. Okun, \emph{Weighted $\ltwo$-cohomology of Coxeter groups}, Geometry \& Topology 11 (2007), 47-138.

\bibitem{davismoussong}
M. Davis and G. Moussong, \emph{Notes on Nonpositively Curved Polyhedra}, Ohio State Mathematical Research Institute Preprints, (1999).

\bibitem{do2}
M. Davis and B. Okun, \emph{Vanishing theorems and conjectures for the $\ltwo$-homology of right-angled Coxeter groups}, Geometry \& Topology 5 (2001), 7-74.

\bibitem{eckmann}
B. Eckmann, \emph{Introduction to $\ltwo$-Homology}, Notes by Guido Mislin, based on on lectures by Beno Eckmann, (1998).

\bibitem{LL}
J. Lott, W. L\"uck, $\Ltwo$-topological invariants of $3$-manifolds, Invent. Math. 120 (1995) 15-60.

\bibitem{moussong}
G. Moussong, \emph{Hyperbolic Coxeter Groups}. Dissertation, The Ohio State University, 1988.

\end{thebibliography}
\end{document}